\documentclass[12pt]{article}
\usepackage{amsthm}
\usepackage{amsmath}
\usepackage{natbib}
\usepackage[colorlinks,citecolor=blue,urlcolor=blue,filecolor=blue,backref=page]{hyperref}
\usepackage{graphicx}
\usepackage{amssymb}
\usepackage{mathrsfs}
\usepackage{textcomp}
\usepackage{multirow}
\usepackage{float}
\usepackage{array}
\usepackage[T1]{fontenc}
\usepackage[latin9]{inputenc}
\usepackage{setspace}
\usepackage{caption}
\theoremstyle{plain}
\newtheorem{thm}{Theorem}[section]
\theoremstyle{plain}
 \newtheorem{defn}{Definition}[section]
  \theoremstyle{plain}
  
  \theoremstyle{definition}
  
  \theoremstyle{plain}
  \newtheorem{lem}{Lemma}[section]
  \theoremstyle{plain}
  
\numberwithin{equation}{section}
\begin{document}

\title{ \textbf{\large Dirichlet Invariant Process for Spherically Symmetric Distributions}}

\author{\small Reyhaneh Hosseini and Mahmoud Zarepour \thanks{{\em Address for correspondence}: Department of Mathematics and Statistics,
University of Ottawa, Ottawa, Ontario, K1N 6N5, Canada. E-mail: zarepour@uottawa.ca.}}

\maketitle
\pagestyle {myheadings} \markboth {} {Dirichlet Invariant Process for Spherically Symmetric Distributions}
\setlength\parindent{0pt}

\begin{abstract}
\noindent  In this paper, we describe a Bayesian nonparametric approach to make inference for a spherically symmetric distribution.  We consider a Dirichlet invariant process prior on the set of all spherically symmetric distributions and we derive the Dirichlet invariant process posterior. Indeed, our approach is an extension of Dirichlet invariant process for the symmetric distributions on the real line to a spherically symmetric distribution where the underlying distribution is invariant under a finite group of rotations. Moreover, we obtain the Dirichlet invariant process posterior for the infinite transformation group and we prove that it
approaches to Dirichlet process.
\par
\vspace{9pt} \noindent\textbf{Keywords}: Bayesian nonparametric inference; spherical symmetry; Dirichlet invariant process; Dirichlet process.\\
\par
\vspace{-3pt} \noindent\textbf{MSC 2010 subject classifications}: Primary 62G20; secondary 62G10.\\
\end{abstract}

\section{Introduction}
Recently, the Bayesian nonparametric approach in statistical analysis of the problems has become widespread among the scientists. The application of Bayesian nonparametric techniques in model selection and hypothesis testing has been used in various fields of applied sciences like economics, biostatistics and many other areas.
Then, the Bayesian nonparametric inference of the problems requires developing the 
proper prior distributions on infinite dimensional spaces and extraction of the posterior distribution.
In Bayesian nonparametric inference, the Dirichlet process that is introduced by \cite{fb1973} is the most popular prior over the space of all probability measures. Later, \cite{antoniak1974mixtures} extended the Dirichlet process to the mixtures of Dirichlet processes. For some applications of Dirichlet
process, see for example, \cite{neal1992bayesian}, \cite{lo1984class}
and \cite{escobar1995bayesian}. These Bayesian nonparametric priors can be used in the case when the subset of interest is the set of all probability measures on a given probability space.  

However, in many situations, based on our assumptions, we may need to tackle a certain class of priors, i.e., the priors that are focused on smaller subsets. For instance, consider the problem of estimating some functional of a symmetric distribution, i.e., the median, the mode of a symmetric distribution or a test for independence assuming permutation symmetry, exchangeability or spherical symmetry.  Here, we need to use the priors on the class of distributions which are symmetric around arbitrary points.

Therefore, for a valid Bayesian formulation of the problems in this framework, where the underlying distributions are
invariant under a specific group of transformations, employing the priors with invariant sample paths will be more reasonable. 
  \cite{dalal1979dirichlet} introduced a class of random processes related to Dirichlet processes with invariant sample paths called Dirichlet invariant processes. Specifically, let $\left(\mathscr{X},\mathscr{B}(\mathscr{X})\right)$ be any $p$-dimensional
Euclidean space with associated Borel $\sigma$-field. Let $\mathscr{\mathscr{G}}=\{g_{1},...,g_{k}\}$
be any finite group of measurable transformations $\mathscr{X\rightarrow\mathscr{X}}$. \cite{dalal1979dirichlet} constructed a random probability measure such that for any set  $B$, the probability of the events $B$ and $g_{j}(B)=\{g_{j}(x): x \in B\}$, $(j=1,...,k)$ are identical. Then, having some assumptions about the invariant properties of the distribution of interest, we can embed our knowledge as an invariant group in Dirichlet invariant process prior. For a discussion about some basic techniques that can be used to construct such processes, see for example, \cite{wesler1959invariance}. 

Similar to Dirichlet processes, these processes can be used as priors for Bayesian analysis of a wide range of problems. Fore instance,  \cite{yamato1986bayes,yamato1987nonparametric}  derived the Bayes estimates in one sample case by evaluating the expectation of random
functional of a Dirichlet invariant process. \cite{dalal1979dirichlet} also obtained the Bayes estimator for the symmetric distribution with a known center of symmetry. Moreover, \cite{dala1979nonparametric,dalal1980nonparametric} employed the Dirichlet invariant process for estimating the unknown center of symmetry for the symmetric distributions.

Assume one is interested in making a Bayesian nonparametric inference for the set of all probability measures on the real line $\mathbb{R}$ which are symmetric about $\mu$. Many authors discussed the problem of nonparametric inference for a cumulative distribution function,
where there is no assumption of symmetry. For example, the best invariant
estimator for a continuous distribution is derived by \cite{aggarwal1955some}. 
\cite{fb1973} and \cite{phadia1973minimax} studied the Bayes estimators and the minimax estimators for an arbitrary distribution, respectively. 

\cite{dalal1979dirichlet} considered the symmetry as the invariant property by defining the transformation group  $\mathscr{\mathscr{G}}=\{g_{1},g_{2}\}$ with $g_{1}=x$ and $g_{2}=2\mu-x$. Then, by using Dirichlet invariant process as the prior over the set of all univariate symmetric distributions with a known center of symmetry, the corresponding Dirichlet invariant process posterior was computed. 
We extend the Dirichlet invariant process to develop a prior over the set of spherically symmetric distributions and derive the corresponding posterior process for this class of distributions.  As an especial case, we place a Dirichlet invariant process prior over the space of bivariate spherically symmetric distributions, i.e, the underlying distribution is invariant under rotations.  We first begin with a finite group of rotations $\mathscr{G}$ such that $\mid\mathscr{G}\mid=k$. Then, the case when $k\rightarrow\infty$ and we prove that our Dirichlet invariant process approaches to Dirichlet process.

The outline of this paper is organized as follows: In Section 2, we
give an essential background on Dirichlet process, Dirichlet invariant process and some of its properties.
Following this, in Section 3, we consider  Dirichlet invariant process prior for a spherically symmetric distribution and obtain the Dirichlet invariant process posterior for a finite group of rotations. Then, we derive the Dirichlet invariant process posterior for the case where $k$ approaches to infinity. An algorithm for simulation purposes and computational studies is also provided at the end of Section 3. In Section 4, we explain how the approach of Section 3 can be generalized to the distributions in higher Euclidean space and in the final section,
we conclude with a brief discussion.

\section{Dirichlet invariant process and invariant random probability measure}

In this section, we review the definition, construction and various properties of Dirichlet invariant process that will be used later in Section 3. 
Despite the relative complexity of the Dirichlet invariant process compared to the Dirichlet process, most of their properties are similar. We first recall the definition of the Dirichlet process for general Bayesian statistical modeling as a distribution over probability distributions.

\begin{defn}(\cite{fb1973}) Let $\mathscr{X}$ be a set, $\mathscr{A}$ be a
$\sigma$-field of subsets of $\mathscr{X}$, $H$ be a probability
measure on $(\mathscr{X,A})$ and $\alpha>0$. A random probability
measure $P$ is called a Dirichlet
process  with concentration parameter $\alpha$ and the base distribution $H$ denoted by $P\sim DP(\alpha H)$ on $(\mathscr{X,A})$
if for any finite measurable partition $\{A_{1},\ldots,A_{k}\}$ of
$\mathscr{X}$, the joint distribution of the random variables $P(A_{1}),\ldots,P(A_{k})$
is a k-dimensional Dirichlet distribution with parameter $(\alpha H(A_{1}),\ldots,\alpha H(A_{k}))$,
where $k\geq2$.
\end{defn}

We assume that if $H(A_{k})=0$, then $P(A_{k})=0$ with probability
one. The next theorem shows that the Dirichlet process has the conjugacy
property. That is, the posterior distribution given the data is again a Dirichlet process. In the following theorem and throughout this paper, we use a "*" as a superscript to denote posterior quantities.

\begin{thm}
(\cite{fb1973}) Let $X_{1},\ldots,X_{m}$
be an i.i.d. sample from $P\thicksim\textrm{\textrm{\ensuremath{DP}}}(\alpha H)$.
The posterior distribution of $P$ given $X_{1},\ldots,X_{m}$ denoted by $P_{m}^{*}= (P\mid X_{1},\ldots,X_{m})$ is
a Dirichlet process $DP(\alpha H+\overset{m}{\underset{i=1}{\sum}}\delta_{X_{i}})$, where $\delta_{X}(\cdot)$ is the Dirac measure, i.e., $\delta_{X}(A)=1$
if $X\in A$ and $0$ otherwise.
\end{thm}

We can denote the Dirichlet process posterior by $DP(\alpha_{m}^{*}H_{m}^{*})$, where 
\begin{equation}
\alpha_{m}^{*}=\alpha+m
\end{equation}
and 
\begin{eqnarray}
H_{m}^{*} 
= \frac{\alpha H+\overset{m}{\underset{i=1}{\sum}}\delta_{X_{i}}}{\alpha+m}
  =  p_{m}H+(1-p_{m})F_{m},
\end{eqnarray}
with $p_{m}=\frac{\alpha}{\alpha+m}$ and $F_{m}=\frac{1}{m}\overset{m}{\underset{i=1}{\sum}}\delta_{X_{i}}$.
Thus, the posterior base distribution $H_{m}^{*} $ is a mixture of the prior guess $H,$ and
the empirical distribution $F_{m}$. \cite{dalal1979dirichlet} applied the Ferguson's
approach to define the Dirichlet invariant process as   
 a prior over the space of all distributions invariant under a group of transformations. 
\begin{defn}
 (\cite{dalal1979dirichlet}) Let $(\Omega,\mathscr{F},Q)$ be any probability
space and $\mathscr{P}(\mathscr{X})$ 
denote the space of all probability measures on the $p$-dimensional
Euclidean space $\left(\mathscr{X},\mathscr{B}(\mathscr{X})\right)$
with corresponding Borel $\sigma$-field  $\mathscr{B}(\mathscr{P}(\mathscr{X}))$ generated by the weak
topology.  Let $\mathscr{\mathscr{G}}=\{g_{1},...,g_{k}\}$
be any finite group of measurable transformations $\mathscr{X\rightarrow\mathscr{X}}$.  Further let $P:(\Omega,\mathscr{F})\rightarrow\left(\mathscr{P}\mathscr{\left(X\right)},\mathscr{B}\left(\mathscr{P}(\mathscr{X})\right)\right)$
be any measurable mapping. Then, the random probability measure $P$
is a $\mathscr{G}$-invariant random probability measure if $P(A)=P(g(A))$
for all $g$ in $\mathscr{G}$, where $P(A)$ is the probability of
set $A$ under probability measure $P$.
\end{defn}
\begin{defn}
(\cite{dalal1979dirichlet}) An invariant random probability
measure $P$ is a Dirichlet invariant process denoted
by $P\sim DIP(\alpha H)$ if for any $\alpha>0$, there exists a $\mathscr{G}$-invariant
measure $H$ on  $\left(\mathscr{X},\mathscr{B}(\mathscr{X})\right)$
such that for every $\mathscr{G}$-invariant measurable partition
(that is the sets of the partitions are $\mathscr{G}$-invariant and
$\mathscr{B}(\mathscr{X})$ measurable) $B_{1},...,B_{k}$ of $\left(\mathscr{X},\mathscr{B}(\mathscr{X})\right)$,
the joint distribution of $(P(B_{1}),\ldots,P(B_{k}))$ is $Dir\left(\alpha H(B_{1}),\ldots,\alpha H(B_{k})\right)$.
\end{defn}

The proofs of all results in this section are discussed by \cite{dalal1979dirichlet} and are similar to those for Dirichlet process as presented in \cite{fb1973}.  
Notice that $P$ is a Dirichlet invariant process with associated parameters $\alpha$ and $H$ denoted by $P\sim DIP(\alpha H)$. Theorem \ref{invar} shows that similar to Dirichlet process, the conjugacy property holds for Dirichlet invariant process. 

\begin{thm}\label{invar}
(\cite{dalal1979dirichlet}) Let $P\sim DIP(\alpha H)$  and $X_{1},...,X_{m}$ be a sample of
size $m$ from $P$. Then, the conditional distribution of $P$ given
$X_{1},...,X_{m}$  is $DIP(\alpha H+\overset{m}{\underset{i=1}{\sum}}\delta_{X_{i}}^{g})$,
where $\delta_{X_{i}}^{g}=\frac{1}{k}\overset{k}{\underset{j=1}{\sum}}\delta_{g_{j}(X_{i})}$
and $\delta_{X}$ is a measure degenerate at $X$ and $k=|G|$.
\end{thm}

\section{Bayesian inference for spherically symmetric distribution}
In various situations, some evidences show that the underlying distribution
of interest is symmetric. In this case, considering the priors that
impose additional assumption of symmetry 
about the distribution leads to a more reliable and exact estimation
of the parameter of interest. 

\cite{dalal1979dirichlet} considered a special case of this problem when
the unknown underlying distribution $F$ is assumed to be symmetric about a known
center of symmetry $\mu$. Specifically, suppose that $X_{1},\ldots,X_{m}$
is a sample from an unknown distribution $F$ symmetric about $\mu$,
where $\mu$ is known. Here, the set $\mathscr{X}$ is the space of
all distribution functions $F$ on the real line $\mathbb{R}$ and
symmetric about $\mu$. Place a Dirichlet
invariant process $P\sim DIP(\alpha H)$ as the prior
on $F$, where $H$ is a finite measure symmetric about $\mu$. Then,
under the invariant transformation group $\mathscr{G}=\left\{ x,2\mu-x\right\} $,
the Dirichlet invariant process posterior is obtained with the base distribution 
\[
{H}_{m}^{*}=\frac{\alpha H+\frac{1}{2}\overset{m}{\underset{i=1}{\sum}}\left(\delta_{X_{i}}+\delta_{2\mu-X_{i}}\right)}{\alpha+m},
\]
where $F_{m}=\frac{1}{2m}\overset{m}{\underset{i=1}{\sum}}\left(\delta_{X_{i}}+\delta_{2\mu-X_{i}}\right)$
is the $F$-symmetrized version of the empirical distribution function. Further,
from the Glivenko-Cantelli theorem and the fact that $p_{m}=\frac{\alpha}{\alpha+m}\rightarrow0$
as $m\rightarrow\infty$ , it follows that ${H}_{m}^{*}$  converges
to the true distribution function uniformly almost surely. The aim
is to mimic Dirichlet invariant
process for constructing a prior on the space of spherically symmetric distributions
on $\mathbb{R}^{p}$. For simplicity, we consider $p=2$, however, a generalization is discussed in the forthcoming section. Indeed, our approach
is an extension of Dirichlet invariant process for the symmetry 
to the spherical symmetry case. We first recall the definition of the spherically
symmetric distribution. The random vector $\boldsymbol{X}\in\mathbb{R}^{p}$ is spherically symmetric if $\boldsymbol{X}\stackrel{d}=A\boldsymbol{X}$ for all $p\times p$ orthogonal matrices $A$. In other words, $\frac{\boldsymbol{X}}{ \parallel \boldsymbol{X} \parallel}$ is independent from $ \parallel\boldsymbol{X} \parallel$.
\begin{thm}\label{thm}
Let $\left(\mathscr{X},\mathscr{B}(\mathscr{X})\right)$ be any two-dimensional
Euclidean space with associated Borel $\sigma$-field. Let $\mathcal{P}=\left\{ \left(0,\theta_{1}\right],\left(\theta_{1},\theta_{2}\right],\ldots,\left(\theta_{k-1},2\pi\right]\right\}$ be a partition of the interval $[0,2\pi]$ formed by $k$ equally spaced subintervals. Then,   $\mathscr{G}=\{A_{\theta_{1}},A_{\theta_{2}},\ldots,A_{\theta_{k}}\}$
is a finite group of orthogonal matrices corresponding to the partition $\mathcal{P}$, where $\theta_{k}=2\pi$ and  $
A_{\theta_{j}}=\left(\begin{array}{cc}
\cos\theta_{j} & -\sin\theta_{j}\\
\sin\theta_{j} & \cos\theta_{j}
\end{array}\right), j=1,\ldots,k$. Let $\boldsymbol{X}_{1},\ldots,\boldsymbol{X}_{m}$ be a sample from
an unknown bivariate spherically symmetric distribution $F$ on $\left(\mathscr{X},\mathscr{B}(\mathscr{X})\right)$.
Consider a Dirichlet invariant process prior $P\sim DIP(\alpha H)$
on $F$. 
Then, the Dirichlet invariant process posterior is obtained with parameters $\alpha^*_{m}=\alpha+m$ and
\begin{equation}
{H}_{k,m}^{*}=\frac{\alpha H+\frac{1}{k}\overset{k}{\underset{j=1}{\sum}}\overset{m}{\underset{i=1}{\sum}}\delta_{A_{\theta_{j}}\boldsymbol{X}_{i}}}{\alpha+m}.\label{eq:spheri_esti}
\end{equation}
\end{thm}
\begin{proof}
Since the random vector $\boldsymbol{X}$ is spherically symmetric, we have
$\boldsymbol{X}\stackrel{d}=A_{\varTheta}\boldsymbol{X}$,
where 
\[
A_{\varTheta}=\left(\begin{array}{cc}
\cos\varTheta & -\sin\varTheta\\
\sin\varTheta & \cos\varTheta
\end{array}\right)
\]
 for all $\varTheta\sim U[0,2\pi]$. Consider the partition $\mathcal{P}$ and
the invariant group of transformations $\mathscr{G}$. Notice that $A_{2\pi}=I$ and $A_{\theta_{i}}A_{\theta_{j}}=A_{\theta_{i}+\theta_{j}}$.
We place a Dirichlet invariant process prior $P\sim DIP(\alpha H)$
on $F$, where $H$ is a bivariate invariant measure on $\left(\mathscr{X},\mathscr{B}(\mathscr{X})\right)$.
Then, by Theorem \ref{invar}, the conditional distribution of $P$ given $\boldsymbol{X}_{1},...,\boldsymbol{X}_{m}$
 is $DIP(\alpha H+\overset{m}{\underset{j=1}{\sum}}\delta_{\boldsymbol{X}_{i}}^{g})$,
where $\delta_{\boldsymbol{X}_{i}}^{g}=\frac{1}{k}\overset{k}{\underset{j=1}{\sum}}\delta_{g_{j}(\boldsymbol{X}_{i})}$. By the group $\mathscr{G}$, we have $g_{j}(\boldsymbol{X}_{i})=A_{\theta_{j}}\boldsymbol{X}_{i},\,(j=1,\ldots,k)$
for each $\boldsymbol{X}_{i},\,(i=1,\ldots,m)$ and 

\[
\delta_{\boldsymbol{X}_{i}}^{g}=\frac{1}{k}\overset{k}{\underset{j=1}{\sum}}\delta_{g_{j}(\boldsymbol{X}_{i})}=\frac{1}{k}\overset{k}{\underset{j=1}{\sum}}\delta_{A_{\theta_{j}}\boldsymbol{X}_{i}}.
\]
Therefore, we have 

\[
H_{k,m}^{*}=\frac{\alpha H+\frac{1}{k}\overset{k}{\underset{j=1}{\sum}}\overset{m}{\underset{i=1}{\sum}}\delta_{A_{\theta_{j}}\boldsymbol{X}_{i}}}{\alpha+m}.
\]
\end{proof}
The next lemma proves the limiting distribution of Dirichlet invariant
process for any finite Borel sets $B_{1},\ldots,B_{N}\in\mathscr{B}(\mathscr{X})$.
 In this paper, \textquotedblleft{}
$\stackrel{d}{\rightarrow}$\textquotedblright{} denotes the convergence
in distribution.
\begin{lem}
\label{finite_dip}Suppose the assumptions of Theorem \ref{thm} hold. Let $P_{k,m}^{*}\sim DIP(\alpha_{m}^{*}H_{k,m}^{*})$
be the Dirichlet process posterior. Then, as $k\rightarrow\infty$,
for any fixed sets $B_{1},\ldots,B_{N}\in\mathscr{B}(\mathscr{X})$,
\[
\left(P_{k,m}^{*}(B_{1}),\ldots,P_{k,m}^{*}(B_{N})\right)\stackrel{d}{\rightarrow}\left(P_{m}^{*}(B_{1}),\ldots,P_{m}^{*}(B_{N})\right),
\]
 where $P_{m}^{*}\sim DP(\alpha_{m}^{*}H_{m}^{*})$  is Dirichlet
process posterior with parameters $\alpha_{m}^{*}=\alpha+m$ and   \textup{
\[
H_{m}^{*}=\frac{\alpha H+\overset{m}{\underset{j=1}{\sum}}F_{A_{\Theta}\boldsymbol{X}_{j}}}{\alpha+m}.
\]
}
\end{lem}
\begin{proof}
Without loss of generality, let $N=2$. Consider any arbitrary partition $\left\{ B_{1},B_{2}\right\} $ of
the space $\mathscr{X}$. Note that 
\begin{eqnarray*}
 &  & \left(P_{k,m}^{*}(B_{1}),P_{k,m}^{*}(B_{2})\right)\\
 & \sim & Dir\left(\alpha_{m}^{*}H_{k,m}^{*}(B_{1}),\alpha_{m}^{*}H_{k,m}^{*}(B_{2}),\alpha_{m}^{*}\left(1-H_{k,m}^{*}(B_{1})-H_{k,m}^{*}(B_{2})\right)\right)
\end{eqnarray*}
Set $Y_{k,i}=P_{k,m}^{*}(B_{i})$ and $v_{k,i}=H_{k,m}^{*}(B_{i}),\,i=1,2$.
Then, the joint density function of $\boldsymbol{Y}=(Y_{k,1},Y_{k,2})$
is 
\begin{eqnarray*}
f_{Y_{k,1},Y_{k,2}}(y_{k,1},y_{k,2}) & = & \frac{\Gamma\left(\alpha_{m}^{*}\right)}{\Gamma(\alpha_{m}^{*}v_{k,1})\Gamma(\alpha_{m}^{*}v_{k,2})\Gamma(\alpha_{m}^{*}(1-v_{k,1}-v_{k,2}))}y_{k,1}^{\alpha_{m}^{*}v_{k,1}-1}y{}_{2}^{\alpha_{m}^{*}v_{k,2}-1}\\
&&\times(1-y_{k,1}-y_{k,2})^{\alpha_{m}^{*}(1-v_{k,1}-v_{k,2})-1}.
\end{eqnarray*}

By using the Scheff\'e's lemma, we only need to prove that
\begin{eqnarray*}
f_{Y_{k,1},Y_{k,2}}(y_{k,1},y_{k,2})& \rightarrow & \frac{\Gamma\left(\alpha_{m}^{*}\right)}{\Gamma(\alpha_{m}^{*}t_{1})\Gamma(\alpha_{m}^{*}t_{2})\Gamma(\alpha_{m}^{*}(1-t_{1}-t_{2}))}y_{1}^{\alpha_{m}^{*}t_{1}-1}y{}_{2}^{\alpha_{m}^{*}t_{2}-1}\\
&&\times(1-y_{1}-y_{2})^{\alpha_{m}^{*}(1-t_{1}-t_{2})-1}
\end{eqnarray*}
as $k\rightarrow \infty$, where $t_{i}=H_{m}^{*}(B_{i}),\, i=1,2$.  We need to find
$
\underset{k\rightarrow\infty}{\lim}f_{Y_{k,1},Y_{k,2}}(y_{k,1},y_{k,2}).$

Therefore, we only need to show that $v_{k,i}=H_{k,m}^{*}(B_{i}) \rightarrow H_{m}^{*}(B_{i})=t_{i}$
as $k\rightarrow\infty$. We have 
\[
\underset{k\rightarrow\infty}{\lim}H_{k,m}^{*}=\frac{\alpha H+\overset{m}{\underset{i=1}{\sum}}\underset{k\rightarrow\infty}{\lim}\overset{k}{\underset{j=1}{\sum}}\delta_{A_{\theta_{j}}\boldsymbol{X}_{i}}/k}{\alpha+m},
\]
where the limit holds with respect to the weak topology. Let $G=\left\{ A_{\theta}\mid0\leq\theta\leq2\pi\right\}$
, where $A_{\theta}=\left(\begin{array}{cc}
\cos\theta & -\sin\theta\\
\sin\theta & \cos\theta
\end{array}\right)$. Define 
\[
G_{m}=\left\{ A_{\theta}\mid\theta=\frac{2k\pi}{2^{m}},\,k=0,1,\ldots,2^{m}\right\} ,
\]
where $m\geq1$. Then, $G_{m}$ is a finite subgroup of $G$ and $\mid G_{m}\mid=2^{m}$.
Define $K=\overset{\infty}{\underset{m=1}{\cup}}G_{m}$. Notice that for any $A_{\theta}\in G$, there is a $B\in K$ such that the Euclidean distance 
$\parallel A_{\theta}-B\parallel<\varepsilon$. We have 
\[
[0,2\pi]=\overset{2^{m}-1}{\underset{k=0}{\cup}}\left[\frac{2k\pi}{2^{m}},\frac{2(k+1)\pi}{2^{m}}\right].
\]
Take $m\in\mathbb{N}$ such that $\frac{\pi}{2^{m-1}}<\frac{\varepsilon}{2}$. Hence, there exists an $r$, $(0\leq r\leq2^{m})$ such that $\theta\in\left[\frac{2r\pi}{2^{m}},\frac{2(r+1)\pi}{2^{m}}\right]$.
Now let $B=\left(\begin{array}{cc}
\cos\delta & -\sin\delta\\
\sin\delta & \cos\delta
\end{array}\right)$, where $\delta=\frac{2r\pi}{2^{m}}.$ Then, by using Euclidean norm, we have
\begin{eqnarray}
 \parallel A_{\theta}-B\parallel=\sqrt{2}\sqrt{\left(\cos\theta-\cos\frac{2r\pi}{2^{m}}\right)^{2}+\left(\sin\theta-\sin\frac{2r\pi}{2^{m}}\right)^{2}}\label{eq:metric2}
\end{eqnarray}
On the other hand, 
\begin{equation}
\mid\cos\theta-\cos\frac{2r\pi}{2^{m}}\mid\leq\mid\theta-\frac{2r\pi}{2^{m}}\mid\leq\frac{\pi}{2^{m-1}}<\frac{\varepsilon}{2}.\label{eq:met3}
\end{equation}
 Similarly, $\mid\sin\theta-\sin\frac{2r\pi}{2^{m}}\mid<\frac{\varepsilon}{2}.$
Then, by (\ref{eq:metric2}), 
\[
\parallel A_{\theta}-B\parallel<\varepsilon.
\]
Therefore, for any closed set C from strong law of large numbers we have
\[
\underset{k\rightarrow\infty}{\lim}\overset{k}{\underset{j=1}{\sum}}\delta_{A_{\theta_{j}}\boldsymbol{X}_{i}}(C)/k=E\left(\delta_{A_{\Theta}\boldsymbol{X}_{i}}(C)\right)=F_{A_{\Theta}\boldsymbol{X}_{1}}(C).
\]
Hence, 
\[
\overset{k}{\underset{j=1}{\sum}}\delta_{A_{\theta_{j}}\boldsymbol{X}_{i}}/k\stackrel{d}{\rightarrow} F_{A_{\Theta}\boldsymbol{X}_{1}},
\]
where $F_{A_{\Theta}\boldsymbol{X}_{1}}(\cdot)$ is the probability measure for the random variable $A_{\Theta}\boldsymbol{X}_{1}$.
\end{proof}
Lemma \ref{finite_dip} proves that the finite-dimensional distributions
of the process $ P_{k,m}^{*}$ converge to the corresponding finite-dimensional
distribution of  $P_{m}^{*}$. 
In the following theorem, we obtain 
the Dirichlet invariant process posterior for the infinite group, i.e., when $\mid\mathscr{G}\mid=k\rightarrow\infty$. Indeed, we prove that the Dirichlet invariant process approaches to Dirichlet process as $k\rightarrow\infty$ with respect to the weak topology.
\begin{thm}
Suppose the assumptions of Theorem \ref{thm} hold.
Consider a Dirichlet invariant process prior $P\sim DIP(\alpha H)$
on $F$ and denote the corresponding posterior Dirichlet invariant
process by $P_{k,m}^{*}\sim DIP(\alpha_{m}^{*}H_{k,m}^{*})$. Then, as
$k\rightarrow\infty$, 
\[
P_{k,m}^{*}\stackrel{d}{\rightarrow}P_{m}^{*}
\]
\end{thm}

\begin{proof}
The finite dimensional convergence is proved in Lemma \ref{finite_dip}. Now, let $\beta=\beta_{1}+\beta_{2}>1$, $\gamma=\gamma_{1}+\gamma_{2}>0$ and $\mu$ be a finite nonnegative measure on $T=[0,1]\times[0,1]$. Consider two neighboring blocks $C=(s,t]\times (a,b]$ and $D=(t,u]\times (a,b]$ in $T$, where $s\leq t\leq u$. By \citep{bickel1971convergence} and Theorem 13.5 of \citep{billingsley2013convergence}, 
we need to prove that 
\[
E\left[( P_{k,m}^{*}(C))^{\gamma_{1}}(P_{k,m}^{*}(D))^{\gamma_{2}}\right]\leq(\mu(C))^{\beta_{1}}(\mu(D))^{\beta_{2}},
\]

where $\beta_{1}$, $\beta_{2}$, $\gamma_{1}$ and $\gamma_{2}$ satisfy $\beta_{1}+\beta_{2}>1$ and $\gamma_{1}+\gamma_{2}>0$. Take $\beta_{1}=\beta_{2}=\gamma_{1}=\gamma_{2}=1$ and let $\mu(\cdot)=\frac{\alpha \lambda(\cdot)+m}{\alpha+m}$, where $\lambda$ is the Lebesgue measure on $T=[0,1]\times[0,1]$. Then, it is enough to show that 
\begin{eqnarray*}
& & E\left[P_{k,m}^{*}(C) P_{k,m}^{*}(D)\right]\leq\\
& &\left(\frac{\alpha}{\alpha^{*}}\right)^{2}\left((b-a)^{2}(t-s)(u-t)+\frac{m}{\alpha}(b-a)(u-s)+\left(\frac{m}{\alpha}\right)^{2}\right).
\end{eqnarray*}
Without loss of generality, we only consider the 
distribution $H$ on $T=[0,1]\times[0,1]$, where $H((u_{1},u_{2}]\times (w_{1},w_{2}])=(u_{2}-u_{1})(w_{2}-w_{1})$. Note that
\begin{eqnarray*}
 &  & \left(P_{k,m}^{*}(C),P_{k,m}^{*}(D)\right)\\
 & \sim & Dir\left(\alpha_{m}^{*}H_{k,m}^{*}(C),\alpha_{m}^{*}H_{k,m}^{*}(D),\alpha_{m}^{*}\left(1-H_{k,m}^{*}(C)-H_{k,m}^{*}(D)\right)\right)
\end{eqnarray*}
and 
\[
 E\left[ P_{k,m}^{*}(C)P_{k,m}^{*}(D)\right]=\frac{\alpha^{*}}{\alpha^{*}+1}H_{k,m}^{*}(C)H_{k,m}^{*}(D).
\]
We have 
\begin{eqnarray*}
&&H_{k,m}^{*}(C)H_{k,m}^{*}(D) \leq\\
& &\left(\frac{\alpha}{\alpha^{*}}\right)^{2}\left((b-a)^{2}(t-s)(u-t)+\frac{m}{\alpha}(b-a)(u-s)+\left(\frac{m}{\alpha}\right)^{2}\right).\\
\end{eqnarray*}
Then, 
\begin{eqnarray*}
& & E\left[ P_{k,m}^{*}(C) P_{k,m}^{*}(D)\right]\leq\\
& &\left(\frac{\alpha^{*}}{\alpha^{*}+1}\right)\left(\frac{\alpha}{\alpha^{*}}\right)^{2}\left((b-a)^{2}(t-s)(u-t)+\frac{m}{\alpha}(b-a)(u-s)+\left(\frac{m}{\alpha}\right)^{2}\right).
\end{eqnarray*}
\end{proof}
 In the following, we provide an algorithm which can be used in computational studies. Let $\boldsymbol{X}_{1},\ldots,\boldsymbol{X}_{m}$
be a sample distributed with an unknown bivariate spherically symmetric distribution $F$. Consider a Dirichlet invariant process with parameters $\alpha$ and $H$, where $H$ is a bivariate spherically symmetric distribution. The algorithm below explains the steps of generating samples from a Dirichlet invariant process posterior $DP(\alpha^{*}_{m}H^{*}_{k,m})$. Let $k$ be a large positive integer. Then,
\begin{enumerate}
\item Consider the jump points $\theta_{i}=2\pi i/k, (i=0,\ldots,k-1)$ with equal
jump sizes from uniform distribution $U[0,2\pi].$
\item For each vector  $\boldsymbol{X}_{i},\,(i=1,\ldots,m)$, compute the values
$A_{\theta_{j}}\boldsymbol{X}_{i},\,(j=1,\ldots,k)$, where $A_{\theta_{j}}=\left(\begin{array}{cc}
\cos\theta_{j} & -\sin\theta_{j}\\
\sin\theta_{j} & \cos\theta_{j}
\end{array}\right).$
\item Compute the empirical distribution of $A_{\theta_{j}}\boldsymbol{X}_{i}$'s,
where $(j=1,\ldots,k)$ and $(i=1,\ldots,m)$ by assigning the equal weight
$\frac{1}{km}$ to each vector $A_{\theta_{j}}\boldsymbol{X}_{i}$. 
\item Generate $N$ observations $\theta_{1}^{*},\ldots,\theta_{N}^{*}$
from $H_{k,m}^{*}$, where $N$ is large enough and 
\[
H_{k,m}^{*}=\frac{\alpha H+\frac{1}{k}\overset{k}{\underset{j=1}{\sum}}\overset{m}{\underset{i=1}{\sum}}\delta_{A_{\theta_{j}}\boldsymbol{X}_{i}}}{\alpha+m}.
\]
\item Generate a sample $P$ from Dirichlet invariant process
posterior with parameters $\alpha^{*}_{m}=\alpha+m$ and $H^{*}_{k,m}$.
\end{enumerate}

The step 5 of the algorithm above can be computed by the algorithm provided in \cite{zarepour2012rapid} 
 for generating samples from Dirichlet process posterior. 

\section{Symmetry in High-Dimensional Euclidean Space}

The approach of Section 3 can be generalized to higher
dimensional Euclidean space.  Consider the distribution which is symmetric about one of the axes of a coordinate system. Then the rotation is performed around a certain axis. The following
three basic orthogonal matrices $A_{\theta_{x}}$, $A_{\theta_{y}}$ and $A_{\theta_{z}}$ rotate vectors counterclockwise by angles $\theta_{x}$, $\theta_{y}$ and $\theta_{z}$
about the x, y, or z-axis, respectively, in three dimensional space.  
\[
A_{\theta_{x}}=\left[\begin{array}{ccc}
1 & 0 & 0\\
0 & \cos\theta_{x} & -\sin\theta_{x}\\
0 & \sin\theta_{x} & \cos\theta_{x}
\end{array}\right]
,
A_{\theta_{y}}=\left[\begin{array}{ccc}
\cos\theta_{y} & 0 & \sin\theta_{y}\\
0 & 1 & 0\\
-\sin\theta_{y} & 0 & \cos\theta_{y}
\end{array}\right]
\]
\[
A_{\theta_{z}}=\left[\begin{array}{ccc}
\cos\theta_{z} & -\sin\theta_{z} & 0\\
\sin\theta_{z} & \cos\theta_{z} & 0\\
0 & 0 & 1
\end{array}\right]
\]
Then, other rotation matrices can be obtained from these three using matrix multiplication. For example, the product
$$
A_{\theta_{x}}A_{\theta_{y}}A_{\theta_{z}}~~~~~~~~~~~~~~~~~~~~~~~~~~~~~~~~~~~~~~~~~~~~~~~~~~~~~~~~~~~~~~~~~~~~~~~~~~~~~~~~~~~~~~~~~~~~~~~~~~~~~~~~~$$$$ =
\small{\left[\begin{array}
{ccc}
\cos\theta_{y}\cos\theta_{z} & -\cos\theta_{y}\sin\theta_{z} & \sin\theta_{y}\\
\cos\theta_{x}\sin\theta_{z}+\sin\theta_{x}\sin\theta_{y}\cos\theta_{z} & \cos\theta_{x}\cos\theta_{z}-\sin\theta_{x}\sin\theta_{y}\sin\theta_{z} & -\sin\theta_{x}\cos\theta_{y}\\
\sin\theta_{x}\sin\theta_{z}-\cos\theta_{x}\sin\theta_{y}\cos\theta_{z} & \sin\theta_{x}\cos\theta_{z}+\cos\theta_{x}\sin\theta_{y}\sin\theta_{z} & \cos\theta_{x}\cos\theta_{y}
\end{array}\right]}
$$

represents a rotation whose yaw, pitch, and roll angles are $A_{\theta_{x}}$, $A_{\theta_{y}}$ and $A_{\theta_{z}}$, respectively. Therefore, the equation (\ref{eq:spheri_esti}) can be generalized based on the orthogonal matrix corresponding to the angles of the rotation for any multidimensional spherically symmetric distribution. Also, the algorithm in Section 3 can be generalized to the multidimensional case easily. 
 
\section{Notes and remarks}
Assuming orthogonal symmetry, the empirical part of the base distribution of the posterior will be augmented by uncountable many points. In regular symmetry for univariate case, the empirical part includes points of the form $\{X_{i}: i=1,2,\ldots,m\}\cup \{2\mu-X_{i};i=1,2,\ldots,m\}$. Therefore, the assumption of spherical symmetry imposes much greater impact in making inference compare to the regular symmetry. 
On the other hand, it seems that the general group requirement can be relaxed. For example, in the case where $X$ has an unknown distribution $F$ with $E(X)=0$, the Dirichlet invariant process posterior should be obtained using the base distribution 
\[
{H}_{m}^{*}=\frac{\alpha H+\overset{m}{\underset{i=1}{\sum}}\delta_{X_{i}-\overline{X}}}{\alpha+m}.
\]

\section{ Discussion}
In this paper, in order to make a Bayesian nonparametric inference for bivariate spherically symmetric distributions, we placed a Dirichlet invariant process prior on the space of this class of distributions and derived the corresponding posterior process. Indeed, our approach is an extension of Dirichlet invariant process to bivariate spherically symmetric distributions. We first obtained the Dirichlet invariant process posterior for a finite group of invariant transformations. Then, we considered an infinite group of transformations.
We proved that for an infinite group, the Dirichlet invariant process approaches to  a Dirichlet process.  The approach developed here can be applied for other forms of invariance. Further, the idea can be used for estimating a density function, a unimodal density function, the mode of a distribution and performing the goodness-of-fit tests. 

\section*{Acknowledgements}
This research was supported by grant funds from the Natural Science
and Engineering Research Council of Canada.

\bibliographystyle{apa}
\bibliography{BIB}

\end{document}